\documentclass[final,5p,times,twocolumn]{elsarticle}
\usepackage{amssymb}

\journal{Systems and Control Letters}
\usepackage[utf8]{inputenc}
\usepackage{hyperref}
\usepackage{amsmath}
\usepackage{amsthm}
\usepackage{amssymb}
\usepackage{latexsym}

\makeatletter
\usepackage{graphicx}
\usepackage{caption}

\newcommand{\X}{\mathcal{X}}
\newcommand{\Y}{\mathcal{Y}}
\newcommand{\R}{\mathbb{R}}
\newcommand{\ep}{\epsilon}
\newcommand{\F}{\mathcal{F}}
\newcommand{\N}{\mathcal{N}}
\newcommand{\C}{\widetilde{C}}
\newcommand{\0}{\textbf{0}}

\usepackage{mathrsfs}
\usepackage{color}

\definecolor{dmagenta}{rgb}{.4,.1,.5}
\definecolor{dblue}{rgb}{.0,.0,.5}
\definecolor{mblue}{rgb}{.0,.0,.7}
\definecolor{ddblue}{rgb}{.0,.0,.4}
\definecolor{dred}{rgb}{.4,.0,.0}
\definecolor{mred}{rgb}{.5,.0,.0}
\definecolor{dgreen}{rgb}{.0,.5,.0}
\definecolor{Eeom}{rgb}{.0,.0,.5}
\definecolor{cm}{cmyk}{1,.0,.0,.0}

\usepackage{graphicx}

\usepackage{xcolor}

\newtheorem{theorem}{Theorem}[section]
\newtheorem{lemma}[theorem]{Lemma}

\theoremstyle{definition}

\theoremstyle{remark}
\newtheorem{remark}[theorem]{Remark}

\numberwithin{equation}{section}

\begin{document}

\title{A dynamic view of some anomalous phenomena in SGD}



\author{Vivek S. Borkar}
\address{Department of Electrical Engineering,
Indian Institute of Technology Bombay,
Mumbai 400076, India,
Email: borkar.vs@gmail.com}


\begin{abstract}
It has been observed by Belkin et al.\  that over-parametrized neural networks exhibit a `double descent' phenomenon. That is,  as the model complexity (as reflected in the number of features)  increases, the test error initially decreases, then increases, and then decreases again.  A counterpart of this phenomenon in the time domain has been noted in the context of epoch-wise  training, viz., the test error decreases with the number of  iterates, then increases, then decreases again. Another anomalous phenomenon is that of \textit{grokking} wherein two regimes of descent are interrupted by  a third regime wherein the mean loss remains almost constant. This note presents a plausible explanation for these and related phenomena by using the theory of two time scale stochastic approximation, applied to the continuous time limit of the gradient dynamics. This gives a novel perspective for  an already well studied theme.
\end{abstract}

\maketitle

\noindent \textit{Key words:} stochastic gradient descent; temporal double descent; grokking; overparametrized neural networks; stochastic approximation; singularly perturbed differential equations; two time scales

\section{Introduction}

Many anomalous phenomena regarding the temporal evolution of stochastic gradient descent (SGD) as applied to over-parametrized neural networks have been pointed out in literature. We specifically consider the following:\\

\noindent 1.\ \textit{Temporal double descent}: Beginning with Belkin et al.\ \cite{Belkin1}, the phenomenon of `\textit{double descent}'  in the training of over-parametrized neural networks using stochastic gradient descent (SGD) has been flagged and extensively studied from various angles \cite{Abascal, Belkin0, Belkin2, Cherkassky, dascoli, Hastie, Kuz, Lafon, Mei, Schaeffer}. See also some generalizations  such as \cite{Belkin0, dascoli}. (See \cite{Loog} for some pre-history.) The most common formulation has been in terms of increasing model complexity as reflected in an increasing basket of features leading to increasing number of parameters. Simply put, as the number of parameters increases, the test error for SGD first decreases, then increases, then decreases again. {\color{violet} However, there has also been a parallel, albeit smaller, literature on \textit{temporal} double descent, i.e., double descent as the iterate count evolves.  This has already been introduced in the context of epoch-wise gradient descent \cite{Davies, Heckel, Nakkiran, Olmin, Pez, Stephenson}. See Figure 1 for a schematic. The point of view taken here is that, while the training error will monotonically decrease to zero, if at various stages of the iteration, we test the learned model on an independent set of test samples, then the corresponding test error can exhibit the temporal double descent.}\\

\begin{figure}[h!]
\begin{center}
\includegraphics[scale=0.7]{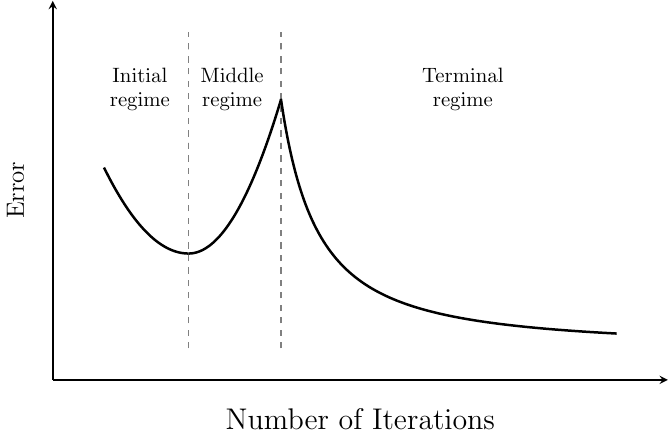}
\caption{}
\end{center}
\end{figure}

\noindent 2.\ \textit{Grokking:} This refers to the phenomenon wherein between two regimes of descent, the SGD exhibits a nearly flat patch for a significant duration \cite{Liu, Mallinar, Power, Varma}. This raises serious issues about the choice of stopping criteria.  See Figure 2 for a schematic.\\

\begin{figure}[h!]
\begin{center}
\includegraphics[scale=0.7]{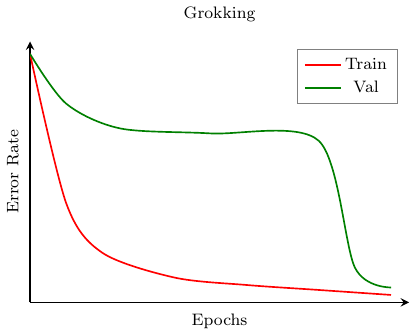}
\caption{}
\end{center}
\end{figure}

\noindent 3.\ \textit{Catapults and language models:} These are additional phenomena where anomalous behaviour is exhibited, possibly repeatedly. The first is the phenomenon wherein the SGD exhibits occasional spikes \cite{Zhu1, Zhu2}. This may be viewed as a repeated occurrence of the ascent as in the first bullet above. The second phenomenon is the training of $n$-gram models of  languages using transformers, where one observes stage-wise decreasing test error analogous to grokking \cite{Varre}.
\begin{figure}[h!]
\begin{center}
\includegraphics[scale=0.7]{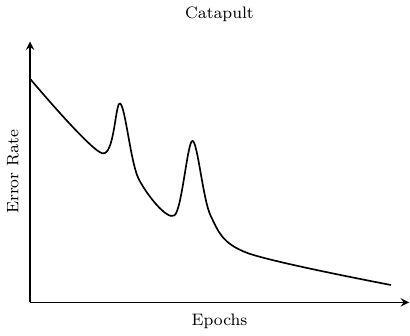}
\caption{}
\end{center}
\end{figure}

The aim of this note is to consider an alternative paradigm for grokking and temporal double descent based on the dynamics of the corresponding SGD as time progresses. The novelty in this work is that we view the iterates as a two time scale stochastic approximation. This two time scale structure is not imposed arbitrarily, but is argued from the well known fact that in very high dimensions, a Lipschitz function will depend mostly on a fraction of the components of the argument, depending only weakly on the rest. This has been made mathematically precise by \cite{Austin}. We recall the key result therein and then use it to motivate a simple model of this phenomenon. As a natural consequence, it is observed that constant step size SGD for such functions naturally splits into two separate, albeit interacting iterations whose effective time scales are separated.\\

Specifically, we consider the loss as a function of the number of iterates of the SGD, equivalently, the number of training samples. The structure of the loss function is exploited as above to justify the time scale separation. This in turn leads to two clearly defined regimes wherein one expects descent due to distinct mechanisms and on different time scales. The two regimes are distinguished by which of the two timescales dominates the evolution of the algorithm, the other one being either too slow or too fast. That leaves an in-between regime where the time scales of the two components are comparable and cannot be separated. This regime is typically not analyzed in the traditional asymptotic analysis. It is argued that this is the regime when  the anomalous behaviour such as ascent or a flat patch  takes place. Some of the aforementioned works already have similar ideas. Our aim here is to give a novel mathematical formulation in the framework of stochastic approximation. \\


%
%

The two time scale dynamics is derived in the next section, following which we recall the basics of two time scale stochastic approximation in Section 3. This framework is then used in Section 4 to identify three `regimes' for SGD with a constant stepsize in an over-parametrized framework, leading to an explanation of the grokking and temporal double descent phenomena. This is done assuming unique minimizers wherever the minimization operation occurs, this being the simplest case to analyze and a building block for more general scenarios. Section 5 describes the general case where this assumption is relaxed.  The final section concludes with brief comments.\\

\noindent \textbf{\large NOTATION:} 
\begin{enumerate}
\item We shall denote by:

1.\ $\0$ the vector of all zeros, of appropriate dimension depending on the context, and,

2.\  by $\Gamma$ a matrix of all zeros of appropriate dimensions depending on the context. 

\item We abbreviate `ordinary differential equation' as `ODE'. 

\item For a deterministic or random sequence $\{w_n\}$, $w_n = O(n)$ will stand for $\limsup_{n\to\infty}\frac{|w_n|}{n} < \infty$ (`a.s.' in the random case). If $w_n \neq 0$, we denote by $w_n = \Theta(n)$ the conditions $w_n = O(n)$ \textit{and} $w^{-1}_n = O(n)$ simultaneously  (`a.s.' in the random case).  For vector-valued $\{w_n\}$, $w_n = O(n)$ or $\Theta(n)$ if it is so componentwise. 

\item For Polish spaces\footnote{A Polish space is a separable topological space whose topology is compatible with a complete metric.} $\X, \Y$, we denote by $C(\X), C(\X;\Y)$ the spaces of  continuous functions $\X \to \R$, resp., $\X \to \Y$. $C_b(\X)$ will denote the space of bounded  continuous functions $\X \to \R$. For compact $\X$, continuous real valued functions on $\X$ are always bounded, so the subscript `$b$' may be dropped.

\item $\nabla \ell$ (resp., $\nabla_x \ell, \nabla_y \ell$) will denote:

1.\  the gradient (resp., partial gradient w.r.t.\ $x$, resp., $y$) if $\ell : \R^k \to \R, k \geq 1$, is a differentiable function, and,

2.\ the Jacobian matrix (partial Jacobian matrix w.r.t.\  $x$, resp., $y$) if $\ell : \R^k \to \R^l, k \geq 1, l \geq 2$.

These are viewed as resp., column vectors and matrices of appropriate dimensions.

\end{enumerate}

\section{A theorem of Tim Austin}

Our starting point is the oft observed fact that functions of a large number of variables typically depend predominantly on a significantly smaller number of variables. In fact, some rigorous statements along these lines are possible when a Lipschitz condition or a bound on the modulus of continuity is imposed on the function \cite{Austin}. 
We state here an instance of the results of Tim Austin about functions on high dimensional spaces (Theorem 1.3, p.\ \cite{Austin}). Let $(S, d, \mu)$ be a connected and locally connected metric probability space. Define the `canonical' random variables on $S^N, N \geq 1$, endowed with its product Borel $\sigma$-field, by : 
$X_i(\omega) = \omega_i$ for $\omega = [\omega_1, \cdots , \omega_N] \in S^N$. Endow $S^N$ with the metric
$$d_N([x_1,\cdots,x_N],[y_1,\cdots,y_N]) \ := \ \frac{1}{N}\sum_{i=1}^Nd(x_i, y_i).$$

\begin{theorem}\label{Austin} Let $m : [0,\infty] \to [0,\infty]$  be non-decreasing. Then for every $\ep > 0$, there is some integer $p \geq 1$, depending on $S, \ep$ and $m$, with the following property. For every $N \geq 1$, if $f : S^N\to \R$ has modulus of continuity at most $m$ for the metric $d^N$, then there is $B \subset  \{1, \cdots , N\}$ with $|B|\leq p$ such that
\begin{equation}
\|f - E\left[f(X) | X_i,i\in B \right]\|_1 < \ep \label{AustinBound}
\end{equation}
where $\|\cdots\|_1$ is the $L_1$-norm on $(S,d,\mu)$.
\end{theorem}

%

This encapsulates the well known  observation that a function of a very large number of variables will typically depend predominantly on a significantly smaller subset thereof. In view of these observations, our starting point will be a stylized model for stochastic gradient descent (SGD) in high dimensions that one encounters while training large over-parametrized neural networks. Consider a function $F: \R^N \to \R$ for some $N \gg 1$, which will serve as our surrogate for the map that maps the parameters to the expected loss  for an over-parametrized neural network. Since our interest is in analyzing SGD for this function, we shall assume that it is continuously differentiable.

In view of the foregoing, we assume that $F$ is of the form 
\begin{equation}
F(z) = f(x,\ep y) \ \mbox{for} \ x \in \R^d, y \in \R^s, \label{fF}
\end{equation}
for some $d,s\gg 1$ with $d+s=N$, $z = [x^T : y^T]^T$, and $0 < \ep \ll 1$. This captures the fact that $F$ depends only weakly on the variable $y$. 

\begin{remark} Note that in reality the separation of scales of dependence may not be binary as depicted here. There may be multiple, even a continuum of degrees of dependence. Our aim is only to formulate  a stylized model that presents a simple scenario wherein the temporal anomalies such as grokking can be explained. We shall comment again on this later in this article.\\
Nor do we assign the three (or more) regimes we identify any relative importance in the overall scheme, as this  can depend on the specifics of the problem. For example, if only the strongly dependent variables $x$ matter and the $y$ variables are deemed `spurious', then only the first descent is important and one might consider `early stopping' as in \cite{Nakkiran}. We do not get into these issues here. \end{remark}

\section{High dimensional SGD}

We begin with the derivation of the two time scale SGD model. We have
$$\nabla F(z) = [\nabla_xf(x, \ep y)^T : \ep\nabla_2f(x, \ep y)^T]^T.$$
The SGD iteration for $F$ with a small constant step size $a \in (0,1)$ is given by
\begin{equation}
z(n+1) = z(n) + a\left(-\nabla F(z(n)) + M(n+1)\right), \ n \geq 0. \label{basicSGD}
\end{equation}
Here $M(n), n \geq 1,$ is the standard `\textit{martingale difference noise}',  i.e., a sequence of integrable random variables in $\R^N$ satisfying
$$E\left[M(n+1)|\F_n\right] = \0, \ n \geq 0,$$
where  $\F_n :=$ the $\sigma$-field $\sigma(M(m), 1 \leq m \leq n;  x(0))$ for $n \geq 0$. Partition $M(n) $ as $[M_1(n)^T : M_2(n)^T]^T$ for $n \geq 1$, where $M_1(n), M_2(n)$ denote resp.\ the first $d$ and the last $s$ components of $M(n)$. 
Let $b := \ep a$ where $\ep > 0$ is as above. Observe that since $a < 1$ and $\ep \ll 1$, 
\begin{equation}
b \ \ll \ a \ < \ \sqrt{a}. \label{eps}
\end{equation} 
The iteration \eqref{basicSGD} can be rewritten as the coupled iteration
\begin{eqnarray}
\lefteqn{x(n+1) = x(n) \ + } \nonumber \\
&& \ \ \ \ \ \ \ \ \ \   a\left(-\nabla_1f(x(n),  \ep y(n)) + M_1(n+1)\right), \label{fast} \\
\lefteqn{y(n+1) = y(n) \ + } \nonumber \\
&& \ \ \ \ \ \ \ \ \ \  b\left(-\nabla_yf(x(n),  \ep y(n)) + M_2(n+1)\right). \label{slow}
\end{eqnarray}
Since $a \gg b$, this is recognized as a constant step size counterpart of  two time scale stochastic approximation introduced in \cite{Borkar2T} for decreasing step sizes. See Section 9.4, bullet 4, of \cite{BorkarBook} for an extension to constant step sizes. In the next section, we apply this framework to our problem.\\

\begin{remark} It is important to note here that we are in fact using a common step size $a$. The two time scale structure emerges due to the structure of $F$ which, on differentiation, multiplies the step size of the second iteration for $\{y(n)\}$  by $\ep \ll 1$, causing the effective step size to be $\ep a = b \ll a$. This induces a time scale separation. \end{remark}

\medskip

\section{Two time scale iterations}

The above observation  facilitates analysis of the SGD in \eqref{fast}-\eqref{slow} as a two time scale stochastic approximation, with the difference (as opposed to \cite{Borkar2T}) being that we have constant step sizes $a \gg b > 0$.  In the classical theory of \cite{Borkar2T}, $a,b$ are replaced by decreasing step sizes $a(n), b(n), n \geq 0,$ satisfying the Robbins-Monro conditions $\sum_na(n) = \sum_nb(n) = \infty$ and $\sum a(n)^2, \sum_n b(n)^2 < \infty$,  along with  the additional condition $b(n) = o(a(n))$ as $n\to\infty$.  Treating $a(n)$ as a discrete time step, define continuous piecewise linear interpolations of $\{x(n)\}, \{y(n)\}$ as follows : Let $t(0) = 0$, $t(n+1) = t(n) + a(n)$ for $n \geq 0$. Set $\bar{x}(t(n)) = x(n), \bar{y}(t(n)) = y(n)$ with linear interpolation on $[t(n), t(n+1)], n \geq 0$.  As argued in \cite{Borkar2T} (also \cite{BorkarBook}, Section 8.1), one can show that 
$$\lim_{t_0\to\infty}\sup_{t\in[t_0,t_0+T]}\|(\bar{x}(t), \bar{y}(t)) - (\tilde{x}(t), \tilde{y}(t))\| \to 0 \ \ \mbox{a.s.}$$ where $\tilde{x}(\cdot), \tilde{y}(\cdot)$ satisfy the ODEs
$$\dot{x}(t) = -\nabla_xf(x(t),y(t)), \ \dot{y}(t) = \0, \ t \in [t_0, t_0+T],$$
with $\tilde{x}(t_0) = \bar{x}(t_0), \tilde{y}(t_0) = \bar{y}(t_0)$. This ensures the separation of time scales of \eqref{fast} and \eqref{slow}. 
One can then rigorously argue as in \cite{Borkar2T} or Section 8.1 of \cite{BorkarBook} that the $\{x(n)\}$ iterates track the asymptotic behaviour of the ODE
\begin{equation}
\dot{x}(t) = -\nabla_xf(x(t),  \ep y) \label{fastODE}
\end{equation}
for a slowly varying $y \approx y(t)$. The latter can be taken to be a constant with asymptotically negligible error over the moving time window $[t_0,t_0+T]$, as $t_0\to\infty$. Suppose that for each $y$, $f(\cdot,  \ep y)$ has a unique minimum $\lambda(y)$ that is Lipschitz in $y$. (We relax this assumption in Section \ref{Grooves}.)

\begin{remark} One way this Lipschitz condition can arise is as follows. Suppose that $f(x,\epsilon y)$ is twice continuously differentiable with bounded second partial derivatives. At $x = \lambda(y)$, one has $\nabla_xf(\lambda(y),y) = \0$. Also, $\nabla^2_xf(\lambda(y),y)$ is positive semidefinite. In fact, if $\lambda(y)$ is an isolated minimum, it is generically true that it will be positive definite. (`Generic' in the sense that a small perturbation of the  function will render it positive definite.) We assume it to be positive definite. Then by the implicit function theorem (\cite{Spivak}, pp.\ 41-42), $\lambda(y)$ is continuously differentiable. Furthermore, differentiating the equation $\nabla_xf(\lambda(y),y) = \0$ on both sides  with respect to $y$, we have
$$\nabla_x^2f(\lambda(y),y)\nabla\lambda(y) + \nabla^2_{yx}f(\lambda(y),y) = \Gamma,$$
where $\nabla^2_{yx} = \nabla_{y}\nabla_x.$
Then
$$\nabla\lambda(y) = -\left(\nabla_x^2f(\lambda(y),y)\right)^{-1}\nabla_{yx}^2f(\lambda(y),y),$$
which, under our hypotheses, is bounded in norm. This implies Lipschitz property for $\lambda(\cdot)$ by the mean value theorem. \end{remark}

\medskip

As argued in \cite{BorkarBook}, Section 8.1, $x(n) - \lambda(y(n)) \to 0$ a.s. Suppose, as in \cite{BorkarBook}, Section 8.1,  that the  ODE
\begin{equation}
\dot{y}(t) = -\nabla_yf(\lambda(y(t)),  \ep y(t)) \label{slowODE}
\end{equation}
has a unique globally asymptotically stable equilibrium $y^*$. (This condition will be relaxed in Section \ref{Grooves}.) Then it is shown in \cite{BorkarBook}, Section 8.1,  that 
$\{y(n)\}$ track the asymptotic behaviour of the ODE 
\begin{equation}
\dot{y}(t) = -\nabla_yf(\lambda(y(t)), \ep y(t)) \label{slowODE2}
\end{equation}
and
\begin{equation}
(x(n),y(n)) \to (\lambda(y^*), y^*) \ \mbox{a.s.} \label{convergence}
\end{equation}

\medskip

\begin{remark} In \cite{Borkar2T},  \cite{BorkarBook}, a more general two time scale dynamics is considered, wherein the driving vector fields on both time scales are not necessarily negative gradients.  We have specialized the results therein to the case where on both time scales, it is a gradient descent in the appropriate variable.\end{remark}

The problem of proving that  \eqref{slowODE} converges can be handled as follows. Note that
$$\nabla_yf(\lambda(y), \ep y) = \nabla_y\min_xf(x, \ep y)$$
by the envelope theorem, also known as Danskin's theorem \cite{Danskin}.  (See \cite{Guler} or \cite{Bardi}, pp.\ 42-46 for a modern treatment.) Thus   \eqref{slowODE} is also a gradient descent and will converge to the unique minimum $y^*$ in view of our assumptions. We state this as a lemma.\\

\begin{lemma} As $t\to\infty$, $(x(t),y(t)) \to$ the unique minimum of $F$. \end{lemma}

\begin{remark} The above version of Danskin's theorem requires strong conditions to ensure that $y \mapsto \min_x f(x, \epsilon y)$ is differentiable. More generally, one has a subgradient instead of a gradient on the right hand side, but the above conclusion continues to hold. We stay with the gradient for simplicity. \end{remark}

\medskip

We shall now state the corresponding results for two timescale stochastic approximation with constant step sizes, where one can only expect concentration near the desired limit and not a.s.\ convergence to it.  This is captured by  the results of \cite{BorkarBook}, Section 9.4, bullet 4. In our context, they translate into the following.\\

\begin{theorem} For $x(n), y(n), n \geq 0,$ as above,
\begin{multline*}
	\limsup_{n\uparrow\infty}\left(E\left[\|x(n) - \lambda(y^*)\|^2\right] + E\left[\|y(n) - y^*\|^2\right]\right)  \\ 
	 {} = O(a) + O\left(\frac{b}{a}\right) = O(a) + O(\ep).
\end{multline*}
\end{theorem}

\medskip

See \cite{BorkarBook}, Section 9.4, for details.

\section{The anomalous behaviour}

This classical asymptotic analysis does not explain the anomalous behaviour. For that, we have to dig into the temporal behaviour of the SGD in greater detail, by identifying different regimes in its evolution. We do this next.\\

\begin{enumerate}

\item \textit{The initial regime:} In the initial phase of the iterations, we expect the partial gradients $\nabla_xf(x(n),\ep y(n))$ and $\nabla_yf(x(n),\ep y(n))$ to be comparable and non-negligible, both $x, y$ being away from the corresponding componentwise minimum. The step sizes are constant at $a \gg b > 0$. Thus the first part of the two time scale logic described above indicates that $y(n)$ will hardly change whereas $x(n)$ will quickly start approaching $\lambda(y(n))$. This is the first descent.\\

\item \textit{The terminal regime:} In the third and final phase, $x(n) \approx \lambda(y(n)) =  argmin(f( \cdot , y(n)))$ and hence $\nabla_xf(x(n),y(n)) \approx 0$. But $\nabla_yf(x(n),y(n)) \approx \nabla_y\min_xf(x,y(n))$ remains significant. Hence $y(n)$ will perform the SGD (more generally, stochastic subgradient descent) on 
$$f(\lambda(y),\ep y) = \min_xf(x,\ep y)$$ 
so as to approach $y^*$. The $x(n)$ will correspondingly track $\lambda(y(n))$ (i.e., $x(n) - \lambda(y(n)) \to 0$ a.s.) and approach $\lambda(y^*)$. This is the second and final descent.\\

The above statements amount to the claims from \cite{Borkar2T}, \cite{BorkarBook} specialized to SGD. \\
 
\item \textit{The `middle' regime:} This leaves the in-between regime when the quantities $a\nabla_xf(x(n), \ep y(n))$ and $b\nabla_yf(x(n), \ep y(n))$ are of comparable magnitude and the time scales of descent in the $x$ and $y$ variables cannot be separated. This is a non-asymptotic and hence commonly ignored regime in the theory of two time scale stochastic approximation.
Consider the ODE \eqref{slowODE} 
which  is the dynamics eventually being tracked by the $\{y(n)\}$ iterations in the terminal phase. These iterations in turn are given by
\begin{eqnarray}
\lefteqn{y(n+1) = y(n) - b\Big(\nabla_yf(x(n), \epsilon y(n)) + M_2(n+1)\Big)} \nonumber \\
&=& y(n) - b\Big(\nabla_yf(\lambda(y(n)), \epsilon y(n)) + \nonumber \\
&& [-\nabla_yf(\lambda(y(n)), \epsilon y(n))) + \nabla_yf(x(n), \epsilon y(n))] \nonumber\\
&& + \ M_2(n+1)\Big). \label{spliteq}
\end{eqnarray}
We shall consider this iteration in the middle regime.

\begin{lemma}\label{wedge} Suppose $\nabla^2f$ is uniformly well conditioned in all directions in a neighbourhood of the trajectory of \eqref{slowODE}. Then the landscape in the neighbourhood of the trajectory of  \eqref{slowODE} is `pinched' in the sense that it is slowly varying along the trajectory of \eqref{slowODE}, but rapidly varying in the $x$-direction, with a local minimum in $x$ at $\lambda(y)$. \end{lemma}

\begin{proof}  Because $\lambda(y)$ minimizes $f( \cdot , y)$, we have $f(x,y) \geq f(\lambda(y),Y)$ for $x$ in a neighbourhood of $\lambda(y)$ for fixed $y$.
Also, since $\nabla^2f$ is uniformly well-conditioned in $B$, it follows that $\nabla^2f( \cdot, \ep \times \  \cdot)$ will \textit{not} be so. In particular, the variation along the $y$ direction is $\Theta(\ep)$ times slower than the variation in the $x$ direction at any $y$ on the trajectory of \eqref{slowODE} restricted to $b$.
This completes the proof. \end{proof}


\medskip

Consider the approximation errors between $y(t)$ and the piecewise linear interpolation of $\{y(n)\}$ on an interval $[n_0b, n_0b + T]$ for $t_0 = n_0b$, with $n_0 \geq 0$ and $T = Mb$ for some $M \gg 0$. By standard arguments (see, e.g., Chapter 2, \cite{BorkarBook}), the max norm error between the linear interpolation and the ODE trajectory with a common initial condition  has a contribution from the discretization error that is $O(bT)$. (Without going into the details, it is $O(b^2)$ per iterate and summed over $M$ steps, it is $O(b^2M) = O(bT)$.) The error $\zeta^0_k, n_0 \leq k < n_0 +M,$ due to the martingale noise grows as $O\left(b\sqrt{M\log\log(b^2M)}\right)$ by the law of iterated logarithm for martingales \cite{Heyde}. 
In view of the results of Section 9.2 of \cite{BorkarBook}, we have the mean square error
\begin{equation}
E\left[\|x(n) - \lambda(y(n))\|^2\right]^{\frac{1}{2}} = \sqrt{O(a) + O(\epsilon)} = O(\sqrt{a}). \label{rogue1}
\end{equation}
Since it is weighted by $b$ in each step, the total error is 
$O(b\sqrt{a}M) = O(\sqrt{a}T)$.
Now consider a small neighbourhood $B$ of a point $y(t^*)$ on the trajectory of \eqref{slowODE} and a small interval $[n_0b, n_0b + T]$ as above for a small $T = Mb$. Since this $\Theta(\sqrt{a})$ error is added at each step of \eqref{spliteq} and gets weighted by $b$, the total error over this interval will be $O(\sqrt{a}T)$. Figure 4 shows a schematic for the behaviour of the iteration \eqref{spliteq}, with the dotted line showing a small piece of $y(\cdot)$ satisfying (4.2) and the dark lines the $\{y(n)\}$ given by \eqref{spliteq}.

\begin{figure}[h!]
\begin{center}
\includegraphics[width=3in]{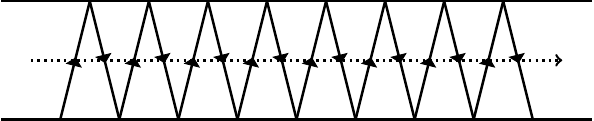}
\caption{}
\end{center}
\end{figure}

That said, Figure 4 is a top down view of the pinched `valley' as viewed from the top and suppresses the vertical dimension. The width of the valley will shrink as the iterates go down to the trajectory of $y(\cdot)$ depicted by the dashed line. The $x$ component, in view of the foregoing, contributes the $\Theta(\sqrt{a})$ jitter (probabilistically speaking) while moving downwards in the valley. As it does so, the valley shrinks in the $x$ direction and eventually its contribution to the error is negligible, which corresponds to the terminal regime. It is the initial stretch, when it is not negligible, 
that we identify with the middle regime. It is clear that the $\Theta(\sqrt{a})$ jitter contributed by the $x$ component, when not negligible, will make a $\Theta(\sqrt{a})$ contribution to the root mean square error, which is significantly larger than the $\approx \Theta(b)$ error contributed by other sources of error in view of \eqref{eps}. This we believe to be the root cause of the presence of the middle region.
\end{enumerate}

\medskip

\begin{remark} One catch in the foregoing argument is that in \eqref{rogue1}, the error is \textit{mean square}, not pathwise. This is one reason we cannot make the above into a `theorem'. This and similar issues (such as the individual roles of the many parameters involved) explain  the wide variation in the occurrence vs non-occurrence of the middle region and the wide variation in its  nature and duration. \end{remark}

This gives an explanation for the anomalous phenomenon for SGD, taking the viewpoint that SGD, like any stochastic approximation, is a noisy discretization of a continuous ODE \cite{BorkarBook}. 

\section{The general case}\label{Grooves}

In this section, we discuss the situation when there are multiple local minima for $f( \cdot , \ep y)$ and / or for $f(\lambda(\cdot), \ep \times \cdot )$, where $\lambda(y)$ is a minimizer of $f( \cdot , \ep y)$. We shall make the additional technical assumption:\\

\noindent \textbf{$(\dagger)$} The map $f$ is twice continuously differentiable and  $f( \cdot , \ep y)$ are real analytic near their local minima.\\

\begin{remark} This excludes connected sets of local minima, the so called `flat' minima, because, e.g.,  a function $\kappa: \R \to \R$ cannot be zero for $x < x^*$ (say) and $=0$ for $x \in [x^*, x^*+\delta)$ for some $\delta > 0$ while being real analytic at $x^*$. Nevertheless, such a situation is non-generic in the sense that it is not stable with respect to small perturbations. The `flat' patches one sees are rarely exactly flat. \end{remark}

In view of the results of \cite{Absil}, this has the following important consequence.\\

\begin{theorem} Every isolated local minimum of $f( \cdot, \ep y)$ is a stable equilibrium of the ODE $\dot{x}(t) = -\nabla_xf(x(t),\ep y)$. \end{theorem}

\medskip

For $z :=[x^T :  y^T]^T$ as before, consider the map 
$$G(z) := \left[\nabla_x f( \cdot , \ep y)^T : y^T\right]^T : \R^N \to \R^N.$$ 
Let $\nabla G: \R^N \to \R^{N\times N}$ denote its Jacobian matrix. Let 
$$\N := \{z \in \R^N : \mbox{Rank}(\nabla G(z)) < N\}.$$ 
Then its image $G(\N)$ has zero Lebesgue measure by the Morse-Sard theorem (\cite{Hirsch}, pp.\ 68-74). In turn, Brouwer's invariance of domain theorem (\cite{Hurewicz}, pp.\ 95-97) then implies that $\N$ is nowhere dense, because if $\N$ contained an open set, that would get mapped to an open set in $F(\N)$, implying positive Lebesgue measure for $F(\N)$, a contradiction. It is easy to see that $\N = \left\{[x^T : y^T]^T : det\left(\nabla^2_xf(x, \ep y)\right) = 0\right\}$, which is clearly closed. Then for $y \not\in$ the closed nowhere dense set $\N$, $\nabla_xf(\cdot , \ep y)$ is non-singular at $x$ and therefore in a neighbourhood of $x$. This  implies that the set of critical points of $f( \cdot , \ep y)$ in $\N^c$ are isolated by the inverse function theorem (\cite{Spivak}, pp.\ 35-39). 
It then follows by the implicit function theorem (\cite{Spivak}, pp.\ 41-42) that given $y$,
$$\C_y := \{[x^T : y^T]^T : \nabla_xf(x, \ep y) = \0\}$$ 
is a $d$-dimensional manifold in a neighbourhood of any $[x^T,y^T]^T \not\in \N$ where $\nabla^2_xF(x,y)$ is non-singular.\\

Then \eqref{fastODE} will be a dynamics on this manifold. There can be several such manifolds, which can be locally extended as long as they remain away from $\N$ where the above application of the inverse function theorem fails. Also, they vary smoothly with $y$, again as long as $\N$ is avoided. Thus the general picture is that of several such manifolds that remain disjoint while away from $\N$ as $y$  changes according to \eqref{slowODE}, but can merge or split on $\N$, a situation akin to bifurcation theory for parametrized differential equations. (See \cite{Berglund} for a related perspective in a different context.) If we assume that away from $N$, every map $y \mapsto \lambda(y)$ satisfies: $f(\lambda(\cdot), \cdot )$ is real analytic near its local minima, then every local minimum in $\N^c$ will be a stable equilibrium for \eqref{slowODE} by \cite{Absil}.\\
 
 A finer analysis of the behaviour of the dynamics away from $\N$ will require the full force of the theory of small noise perturbations of differential equations. We discuss this next.
A common and natural way to address this issue would be through `\textit{small noise limits}'. That is, following a philosophy promoted by Kolmogorov (as reported in \cite{Eckmann}, p.\ 626), we add a small amount of noise to the dynamics. The solutions obtained in the limit as the noise goes to zero are deemed to be the `physical' solutions. (This is a common selection process in many situations with ill-posed problems, one prime example being the viscosity solution concept for Hamilton-Jacobi equations, another is the notion of `stochastically stable Nash equilibria' in evolutionary game theory \cite{Young}.)  However, the addition of noise requires careful modelling. For a single time scale iteration, one may invoke the classical theory of \cite{FW}.  Specifically, for a single ODE $\dot{z}(t) = -\nabla f(z(t))$, the natural small noise perturbation is  
$$dz^\ep(t) = -\nabla f(z^\ep(t))dt + \ep dW(t)$$
where $W(\cdot)$ is a standard brownian motion in $\R^N$. See, \textit{e.g.}, \cite{Berglund}, \cite{Schuss}. The corresponding situation in the two time scale case, however, is more complex and very limited results are available. Here we use the framework of \cite{Athreya}. We briefly describe it below.\\

Following \cite{Athreya}, we shall write the noise perturbed versions of \eqref{fastODE}-\eqref{slowODE} as\footnote{This is equivalent to the dynamics in \cite{Athreya} modulo a time scaling, more convenient for our purposes.} 
\begin{eqnarray}
dx^\ep(t) &=& -\nabla_xf(x^\ep(t), \ep y^\ep(t))dt + s(\ep)dW(t), \label{fastnoise}\\
dy^\ep(t) &=& -\ep \nabla_yf(x^\ep(t), \ep y^\ep(t))dt + \ep^{1+\alpha} dB(t), \label{slownoise}
\end{eqnarray}
for $t \geq 0$, where $W(\cdot), B(\cdot)$ are standard brownian motions in resp., $\R^d$ and $\R^s$, $0 < \alpha < 1$, and $s(\cdot)$ satisfies
\begin{eqnarray*}
s(\ep) \  &\stackrel{\ep\downarrow 0}{\rightarrow}& \ 0, \\
\ && \ \\
s(\ep) &\geq& \sqrt{\frac{K}{\log\left(1 + \frac{1}{\ep}\right)}},
\end{eqnarray*}
for some $K > \frac{2(\Lambda_1 + 2\Lambda_2)}{1 - \alpha}.$
Here $\Lambda_1, \Lambda_2 > 0$ are two constants that depend on the dynamics and need some additional conditions on $f$ - see (11) and (13) in \cite{Athreya}. In this scenario, the distribution of $y(n)$ concentrates on the global minima of $f(\lambda(\cdot),\cdot)$ with relative weights given by Proposition 2.4 d) of \cite{Athreya} (which builds upon the remarkable work of Hwang \cite{Hwang}).
The exact expressions for the relative weights can be found in \cite{Athreya}.\\

This is the `average behaviour' one sees in the $\ep\downarrow 0$ limit. For small $\ep > 0$, the stochastic phenomena are small, but very much present and different. Away from $\N$, the local neighbourhoods of local minima will form a `valley' as $y$ varies. (The terminology is for the case $d=s =1$ which is easy to visualize. We lift it to arbitrary $d,s$.) This is implicit in Lemma \ref{wedge}. What one shall see is that the trajectory follows one valley at a time, with rare jumps between valleys due to low probability noise spikes. This is not an unknown phenomenon and has been extensively analyzed in single time scale dynamics. The mean exit time from a valley is extremely large, being of the order of  $\frac{1}{\ep^2}$ (see Section 6.5 of \cite{FW}). In our set-up of two time scale dynamics, it is not a simple localized  valley as in \cite{FW}, but a parametrized one, thus forming a groove in the landscape. The dynamics then jumps across the grooves at large random  times that have distributions with decay rates $\approx \frac{1}{s(\ep)^2}$ (Section 6.5 of \cite{FW} - we apply this to dynamics in the $x$-variable treating the $y$ component as quasi-static.). The grooves, of course, can merge and split in $\N$ as pointed out earlier.\\

In the terminal regime, the analysis is somewhat easier, since we ideally have the flow of a well defined scheme \eqref{slowODE} converging to one of the critical points, in fact to local minima in view of `avoidance of traps' results such as \cite{BorkarTraps}, \cite{Brandiere}, \cite{Pemantle} along with \cite{Absil}. One possible complication is due to non-isolated minimizers of $F(\cdot,y)$, in which case we need to replace the gradient descent by a sub-gradient descent. This is allowed by Danskin's theorem \cite{Bardi}, \cite{Danskin}. The asymptotic distribution will concentrate on the points $(\lambda(y),y)$ that are local minima of the function $F(\lambda(\cdot),\cdot)$, with higher probabilities for the local minima with lower values for $F(\lambda(\cdot),\cdot)$ as shown in \cite{Aziz}, \cite{Hwang}, using the fact that for a gradient descent $\dot{x}(t) = -\nabla\Psi(x(t))$ in $\R^d$, $\Psi(\cdot) : \R^d \to \R$ itself serves as the Freidlin-Wentzell potential for the corresponding small noise perturbation given by the diffusion process $dx(t) = -\nabla\Psi(x(t))dt + \ep dW(t)$ \cite{FW}.\\

For the problematic middle regime, the earlier comments apply valley-wise outside of $\N$.

\section{Conclusions}

We have given a novel explanation of anomalous phenomena for SGD applied to over-parametrized neural networks, such as grokking, temporal double descent and catapults, using a dynamical picture. That said, it remains a stylized version of the actual SGD, nevertheless rich enough to capture the key phenomena. Some limitations are:
\begin{enumerate}

\item We assume a clear separation between the extent of dependence of $F$ on two classes of variables. This may not be so clear cut. There is also a possibility of three or more well marked scales, which can be analyzed along above lines.

\item Due to the nature of Sard's theorem, the results in the general case are only for $\N^c$. The description of `smallness' is topological, not measure theoretic, hence weaker. The set $\N$ cannot necessarily be ruled out or treated as less important, because it is perfectly likely that some or all of the trajectories of \eqref{slowODE} converge to it. In fact the latter may very well be the reason that the above analysis does not capture the observed anomalies of SGD in all their aspects. Nevertheless, we believe it is a useful step in that direction.

\item The small noise limit argument in the last section is handicapped by the fact that very little literature is available regarding this for two time scale iterations, unlike the single time scale case. There is ample opportunity to develop the basic theory itself much further and then apply it to algorithmic issues such as those considered here.

\end{enumerate}
This also opens up the possibility of a finer analysis  using the full force of singularly perturbed differential equations with noise \cite{FW}, \cite{Schuss}. This is left for the future.\\

\noindent \textbf{Acknowledgements} This research was supported by an award from Google Research Asia. The author thanks Prof.\ Mallikarjuna Rao, Prof.\ Parthe Pandit and Satush Parikh for their comments.

\end{document}